\documentclass[preprint]{elsarticle}

\usepackage[utf8]{inputenc} 

\usepackage{amsthm, amsmath,amssymb,physics}

\newtheorem{theorem}{Theorem}[section]
\newtheorem{corollary}[theorem]{Corollary}
\newtheorem{main}{Main Theorem}
\newtheorem{lemma}[theorem]{Lemma}

\theoremstyle{definition}
\newtheorem{definition}[theorem]{Definition}
\newtheorem{remark}[theorem]{Remark}

\newtheorem{example}{Example}

\usepackage{graphicx} 

\begin{document}
\title{Dynamical properties in the axiomatic theory of ordinary differential equations}
\author{Tomoharu Suda\fnref{fn1}}
\ead{tomoharu.suda@riken.jp}
\address{Department of Mathematics, Faculty of Science and Technology,
Keio University, 3-14-1 Hiyoshi, Kohoku-ku, 
Yokohama, JAPAN}
\fntext[fn1]{Present address: RIKEN Center for Sustainable Resource Science, RIKEN, Wako, Japan}
\begin{abstract}
The axiomatic theory of ordinary differential equations, owing to its simplicity, can provide a useful framework to describe various generalizations of dynamical systems. In this study, we consider how dynamical properties can be generalized to this setting. First, we study the generalizations of dynamical properties, such as invariance and limit sets and show that compact weakly invariant sets can be characterized in terms of the corresponding shift-invariant sets in the function space. This result enables us to consider weakly invariant sets in terms of shift-invariant sets. Next, we compare the axiomatic theory of ODE with other formalisms of generalized dynamical systems, namely generalized semiflow and multivalued semigroup. Finally, we introduce the notion of invariant measures and show that we can generalize classical results, such as Poincar\'e recurrence theorem. 
\end{abstract}

\maketitle
\section{Introduction}
 
Generalizations of dynamical systems are useful when considering evolution equations in the context of infinite-dimensional dynamical systems as well as applications such as control systems, differential inclusions, and piecewise-smooth vector fields, as they lack the uniqueness of time evolution \cite{ball1997continuity, melnik1998attractors, caraballo2003comparison, bernardo2008piecewise}. Although continuous-time dynamical systems are generally described using the concept of flow, applying this concept directly to such cases is impossible. Therefore, a generalizing framework is required. Various theories, such as generalized semiflow \cite{ball1997continuity} and multivalued semiflow \cite{melnik1998attractors} have been proposed (a more comprehensive description can be found in the introduction of \cite{ball1997continuity}).
 
When considering generalized dynamical systems arising from generalized ordinary differential equations (ODEs), the axiomatic theory of ordinary differential equations is a particularly interesting framework as it formalizes the conditions that solutions of a well-posed ordinary differential equation satisfy as axioms and considers how subspaces of a function space behave under these axioms. Furthermore, this approach can proceed from the necessary conditions rather than properties that are sufficiently good to construct analogies of ordinary dynamical systems. This feature enables us to obtain results that hold true in a general setting.

There are two main axiomatic theories of ordinary differential equations: Filippov and Yorke \cite{Filippov_1993, filippov_basic, yorke1969spaces}. Filippov's theory has been studied extensively compared to Yorke’s framework. For example, known results of Filippov's theory include generalization of the Poincare-Bendixon theorem \cite{filippov_basic}, the connection with the notion of local dynamical systems \cite{mychka2010space}, and the results on asymptotically autonomous systems \cite{Klebanov1994}. However, in connection with dynamical systems theory, working within Yorke's theory appears easier. In fact, Yorke’s first paper discusses the dynamical properties; moreover, dynamical systems of shifts in function spaces form the basis of the theory. Furthermore, a topological classification of dynamical systems can be obtained when generalizing the notion of dynamical systems using Yorke's framework \cite{suda2022equivalence}.

Therefore, this study discusses how Yorke's axiomatic theory of ordinary differential equations generalizes the properties of dynamical systems, particularly those related to invariance and limit sets. Furthermore, we discuss how Yorke's theory relates to other frameworks, such as generalized semiflow and multivalued semigroups, for generalizing dynamical systems. Then, we define the notion of an invariant measure in the axiomatic theory of ordinary differential equations as an application of the results obtained and show that a generalization of the classical result holds. 

Before we state the main results of this article, we briefly summarize how Yorke's theory describes a generalized dynamical system. Because the solutions of an ordinary differential equation are generally definable only locally, we consider a subspace $S$ of a space of non-continuable partial maps $C_s(\mathbb{R}, X)$ (Definition \ref{def_cs}) as a solution space. The space $S$ can be interpreted as describing a generalized dynamical system on $X$ if $S$ is invariant under the shift $\sigma$ on $C_s(\mathbb{R}, X)$, which is naturally defined. Because $\sigma$ is a flow on $S$ in this case, we consider the dynamics induced by an evaluation map $\pi: S \to X$ defined by $\pi(\phi) = \phi(0)$. For example, a subset $A \subset X$ can be defined as weakly invariant if 
all $x \in A$ has a corresponding map $\phi \in S$ such that $\pi(\phi) = x$ and $\pi\qty(\sigma(t, \phi) )\in A$ for all $t$.

The basic principle, which is stated in the next theorem, is that a compact weakly invariant set in the phase space corresponds to a well-behaved shift-invariant subset in solution space. This is a generalization of classical results known for Bebutov dynamical systems (Theorems 1.20 and 2.20 in \cite{sibirskij1975introduction}). The key assumption here is the validity of the compactness axiom (Definition \ref{def_axioms}). In rough terms, this condition says that the set of all ``solutions" passing a compact set is compact.
\begin{main}
Let $S\subset C_s(\mathbb{R}, X)$ be a shift-invariant set that satisfies the compactness axiom. 
Then, there exists a one-to-one correspondence between the sets
\[
  \{S' \subset S \mid \text{shift invariant, maximal w.r.t. phase space and }\pi(S')\text{ is compact} \}
\]
and
\[
	\{A\subset X\mid A \text{ is weakly invariant and compact} \}.
\]
\end{main}
 According to Theorem A, a limit set in the solution space $S$ induces a weakly invariant set. We use this property to define a limit set $\omega_S(x)$ for point $x$ in the phase space. Under this definition, the classical results can be generalized.
 \begin{main}
Let $S$ be a shift-invariant set that satisfies the compactness axiom. Then, limit sets $\omega_S(x)$ and $\alpha_S(x)$ are closed and weakly invariant for all $x \in X$.
\end{main}
The next theorem states a connection between the axiomatic theory of ODEs, generalized semiflow, and multivalued semigroups. This result implies that Yorke's theory can be applied to the set of all continuous complete trajectories, which plays an important role in the theory of invariant measures for the multivalued semigroup introduced by Moussa \cite{moussa2017invariant}. 
 \begin{main}
Let $V(t)$ be a multivalued semigroup determined by a generalized semiflow $G \subset C(\mathbb{R}_{\geq 0},X)$ on a compact metric space $X$. 
Then, the set $\mathcal{K}$ of all continuous complete trajectories is nonempty and shift-invariant. As a solution space on $\pi\qty(\mathcal{K})$, $\mathcal{K}$ satisfies existence, compactness, and switching axioms.
\end{main}
Using the observation of Theorem C, we can generalize the definition of invariant measures to the axiomatic theory of the ODE. For solution sets in a compact metric space with the compactness axiom, we can generalize classical theorems on the invariant measures of flow. In particular, the next generalization of the topological Poincar\'e recurrence theorem holds true. This will be construed as a justification for our definition of limit sets and invariant measures. 

\begin{main}
Let $\mu$ be an invariant measure of a shift-invariant subset $S \subset C(\mathbb{R}, X)$. Then, the closure of the set of all recurrent points for $S$ is a full-measure set.
\end{main}

The remainder of this paper is organized as follows. Section \ref{prelim} presents preliminary definitions and results. In particular, an outline of Yorke's axiomatic theory of ODE is presented. Section \ref{solsp} discusses the topological properties of a solution space and how the dynamics in the phase space correspond to those in the solution space. The proof of Theorem A is provided in this section. In Section \ref{asym}, we consider the asymptotic properties of the dynamics induced by the solution space. In particular, we introduce the notion of limit set. Section \ref{comp} shows how the axiomatic theory of ODE relates to other existing formalisms, namely generalized semiflow and multivalued semigroup. Finally, Section \ref{inv_meas} introduces the notion of invariant measures and presents some generalizations of the classical results. Section \ref{conc} contains concluding remarks.

\section{Preliminaries}\label{prelim}
In this section, we present some definitions and the basic results used in the discussion.
\subsection{General topology of partial maps}
To introduce Yorke's axiomatic theory of ODE, we first present the basic definitions and properties of partial maps. We consider partial maps with open domains, as studied by Abd-Allah and Brown \cite{abd1980compact}. 
\begin{definition}[Partial maps]
Let $X$ and $Y$ be topological spaces.
A \textbf{partial map} is a pair $(D, f)$ of subsets $D \subset X$ and a map $f: D \to Y.$ For a partial map $(D,f)$, the set $D$ is called the \textbf{domain}, denoted by $\mathrm{dom}\, f$. By an abuse of notation, we denote a partial map $(D, f)$ by $f: X \to Y.$ The \textbf{image} of a partial map is the set $\mathrm{im}\, f := f \left(\mathrm{dom}\,f \right).$

A partial map $f:X \to Y$ is \textbf{ with open domain} if $\mathrm{dom}\,f$ is open and $f: \mathrm{dom}\,f \to Y$ is continuous.
\end{definition}
If space $Y$ is a separable metric space, the topology of the space of all partial maps $C_p(X, Y)$ can be defined by uniform convergence on compact sets. In other words, a sequence of partial maps $\phi_n \in C_p(X, Y)$ converges to $\phi \in C_p(X,Y)$ as $n \to \infty$ if, for each compact $K \subset \mathrm{dom}\, \phi,$ $K \subset \mathrm{dom}\, \phi_n$ for sufficiently large $n$ and $\sup_{x \in K} d(\phi(x), \phi_n(x))$ converges to $0$ as $n \to \infty$. This topology coincides with the compact-open topology defined by subbases $W(K, V) :=\{\phi \in C_p(X, Y) \mid K \subset \mathrm{dom}\, \phi \text{ and } \mathrm{im}\,\phi\subset V\},$ where $K \subset X$ is compact and $V \subset Y$ is open (Lemma 2.3 in \cite{suda2022equivalence}).

The inverse image of a partial map can be defined as follows. Here, we define for all $\phi \in C_p(X,Y)$ and $A \subset Y,$
\[\phi^{-1} (A) := \{x \in X \mid x \in \mathrm{dom}\, \phi \text{ and } \phi(x) \in A\}.\] 
Under this convention, the inverse image of an open set is open.

In what follows, we assume that $X = \mathbb{R}$ and $Y$ is a separable metric space. Because we consider only partial maps with a real variable, we use the letter $X$ for the phase space instead of $Y$. For example, the space of all partial maps is denoted by $C_p(\mathbb{R}, X).$ Here, we note that $C_p(\mathbb{R}, X)$ is second countable if $X$ is a separable metric space (Proposition 2.1 in \cite{alvarez2016topological}), which justifies the use of sequences when considering topological properties.

An important example of a partial map with an open domain is the evaluation map, which is used later.
\begin{definition}
We define $e: \mathbb{R}\times C_p(\mathbb{R}, X) \to \mathbb{R}\times X$ as
\[
	e(t, \phi):=\qty(t, \phi(t))
\]
if $t \in \mathrm{dom}\,\phi.$
\end{definition}
\begin{lemma}\label{lem_eval}
The map $e$ is a partial map with open domain.
\end{lemma}
\begin{proof}
Let $\qty(t, \phi) \in \mathrm{dom}\, e.$ By definition, we have $t \in \mathrm{dom}\,\phi$. Therefore, we can find a compact neighborhood $K$ of $t$ such that $K \subset \mathrm{dom}\,\phi$. This is equivalent to $\phi \in W(K,X).$ As $t \in K,$ $\mathrm{int}\, K \times W(K,X) \subset \mathrm{dom}\,e.$ Therefore, $\mathrm{dom}\,e$ is open.

For the continuity of $e,$ we show that $\qty(t_n, \phi_n) \to\qty(t, \phi)$ implies $\qty(t_n, \phi_n(t_n) )\to \qty(t, \phi(t))$. Let $\delta >0$ be sufficiently small to satisfy $[t-\delta, t+\delta] \subset \mathrm{dom}\, \phi.$ For sufficiently large $n$, we have $[t-\delta, t+\delta] \subset \mathrm{dom}\, \phi_n.$ Consequently, as we have $t_n \to t,$ there exists $N$ such that $[t-\delta, t+\delta] \subset \mathrm{dom}\, \phi_n$ and $t_n \in [t-\delta, t+\delta] $ for all $n \geq N.$ Therefore, we have
\[
\begin{aligned}
		d\qty(\phi_n(t_n), \phi(t)) &\leq d\qty(\phi_n(t_n), \phi(t_n)) +d\qty(\phi(t_n), \phi(t))\\
							&\leq \sup_{s \in [t-\delta, t+\delta]} d\qty(\phi_n(s), \phi(s)) + d\qty(\phi(t_n), \phi(t))
\end{aligned}
\]
if $n \geq N.$
The first term on the right-hand side converges to $0$ as $n \to \infty$ because $\phi_n \to \phi.$ The second term tends to $0$ owing to the continuity of $\phi.$
\end{proof}

By analogy with solutions for an ODE, a natural object of study is the set of all maximally defined partial maps corresponding to non-continuable solutions.
\begin{definition}\label{def_cs}
A partial map $\phi\in C_p(\mathbb{R},X)$ with a nonempty connected domain is \textbf{maximally defined} if, for all $\psi\in C_p(\mathbb{R},X)$ with a nonempty connected domain, the condition $\mathrm{dom}\, \phi \subset \mathrm{dom}\, \psi$ and $\phi=\psi$ on $\mathrm{dom}\, \phi$ implies $\phi=\psi.$

The set of all maximally defined partial maps is denoted by $C_s(\mathbb{R}, X)$.
\end{definition}
The space $C_s(\mathbb{R}, X)$ enjoys properties better than $C_p(\mathbb{R}, X)$ \cite{suda2023equivalence}. For example, $C_s(\mathbb{R}, X)$ is an $T_1$ space. 
\subsection{Axiomatic theory of ordinary differential equations}
As mentioned in Introduction, there are two main formalisms of the axiomatic theory of ODE. Here, we introduce Yorke's theory based on the notion of partial maps with open domains. Details of this theory can be found in \cite{yorke1969spaces, suda2022equivalence}.

The basic idea of Yorke's axiomatic theory of ODE is that a sufficiently well-behaved subspace of the space of partial maps $C_p(\mathbb{R},X)$ can be considered as a solution space for an initial value problem on $X$. 

For example, let us consider an initial value problem for an ODE of the form $\dot x = f(x).$ If the initial condition is $x(0) = x_0,$ the problem here is to find a map $\phi$ in the space of all possible solutions $S =\{\psi \mid \dot \psi(t) = f(\psi(t)) \text{ for all } t \in \mathrm{dom}\, \psi\}$ such that $\phi(0) = x_0.$ In terms of the aforementioned evaluation map $e$, this is equivalent to finding $\phi \in \qty(e|_S)^{-1}(\{(0, x_0)\}),$ where $e|_S: \mathbb{R}\times S \to \mathbb{R}\times X$ is the restriction of $e$ to $\mathbb{R}\times S.$ If this initial value problem is well posed, the space of solutions $S$ has good properties, such as compactness.

In axiomatic treatment, axioms are imposed on the space of solutions.
The axioms are given in terms of the evaluation map as follows:
\begin{definition}\label{def_axioms}
Let $S \subset C_s(\mathbb{R},X)$ and $e_S: \mathbb{R}\times S \to \mathbb{R}\times X$ be the restrictions of $e$ to $\mathbb{R}\times S.$
\begin{enumerate}
	\item The subspace $S$ satisfies the \emph{compactness axiom} if $e_S$ is a proper map.
	\item The subspace $S$ satisfies the \emph{existence axiom} on $W$ if $e_S: e_S^{-1}(W)\to W$ is surjective.
	\item The subspace $S$ satisfies the \emph{uniqueness axiom} on $W$ if $e_S: e_S^{-1}(W)\to W$ is injective.
	\item The subspace $S$ has \emph{domain} $D$ if $e_S$ is defined on $D \times S.$
	
	\end{enumerate}
\end{definition}
\begin{remark}
The axiom statements given here are different from the original ones in \cite{yorke1969spaces}. In Yorke's presentation of axioms, the star construction $S^*W := \{(t, \phi) \mid t \in \mathrm{dom}\,\phi \text{ and } (t, \phi(t))\in W\}$ is used instead of the evaluation map. However, we have $S^*W = e_S^{-1}(W)$, and hence, we can easily see that they are equivalent.
\end{remark}
The compactness axiom plays a central role in the development of the theory as it corresponds to the continuous dependence. 
\begin{theorem}[Theorem 2.13 in \cite{suda2022equivalence}]\label{cptness}
A subset $S \subset C_s(\mathbb{R},X)$ satisfies the compactness axiom if and only if the following property holds: if $(t_n, x_n) \to (t, x)$ is $n \to \infty$ in $\mathbb{R} \times X$ and there exists a sequence of maps $\phi_n \in S$ with $\phi_n(t_n) = x_n,$ there is a map $\psi \in S$ with $\psi(t) = x$ and a subsequence $\{(t_{n_i}, \phi_{n_i})\}$ with $(t_{n_i}, \phi_{n_i}) \to (t, \psi)$ as $i \to \infty.$
\end{theorem}
In applying these axioms to the study of dynamics, we use the shift map on $C_p(\mathbb{R}, X)$ to describe the dynamics.

\begin{definition}
The shift map $\sigma : \mathbb{R} \times C_p(\mathbb{R},X) \to C_p(\mathbb{R},X) $ is defined as
\[
	\sigma(t, \phi)(x) := \phi(x+t),
\]
where $x \in \mathrm{dom}(\phi)-t$.

\end{definition}
\begin{theorem}
  The shift map is continuous and satisfies the following conditions:
\begin{enumerate}
	\item For each $\phi\in C_p(\mathbb{R},X)$ we have $\sigma(0, \phi) = \phi.$
	\item For all $s,t \in \mathbb{R}$ and $\phi\in C_p(\mathbb{R},X)$, we have $\sigma(s,\sigma(t,\phi)) = \sigma(s+t, \phi).$
\end{enumerate}
That is, $\sigma$ is a continuous flow on $C_p(\mathbb{R},X)$.
\end{theorem}
The next result shows how subbasis of $C_p(\mathbb{R}, X)$ is transformed by the shift map.
\begin{lemma}
For all compact subset $K \subset \mathbb{R}$, open subset $V \subset X$ and $t\in \mathbb{R}$, we have
\[
	\sigma\qty(t, W\qty(K, V) ) = W\qty(K-t, V).
\]
\end{lemma}
\begin{proof}
This result follows from a direct calculation.
\end{proof}
For shift-invariant subsets of $C_{s}(\mathbb{R}, X)$, weaker conditions are sufficient to verify that the axioms hold globally. 
\begin{theorem}
Let $S \subset C_{s}(\mathbb{R}, X)$ be shift invariant. Then,
\begin{enumerate}
	\item $S$ satisfies the existence axiom on $\mathbb{R}\times X$ if and only if $S$ satisfies the existence axiom on $\{0\}\times X.$
	\item $S$ satisfies the uniqueness axiom on $\mathbb{R}\times X$ if and only if $S$ satisfies the uniqueness axiom on $\{0\}\times X.$
	\item $S$ has domain $\mathbb{R}$ if and only if $S$ has domain $\{0\}.$ 
\end{enumerate}
\end{theorem}
\begin{proof}
First, we show statement (1).
Let $S$ satisfy the existence axiom on $\{0\}\times X$ and $(s, x) \in \mathbb{R}\times X.$ Then, we can find $\phi \in S$ with $\phi(0) = x.$ Since we have $\sigma(-s, \phi) \in S$ by shift invariance and $\sigma(-s, \phi)(s) = \phi(0) =x, $ $S$ satisfies the existence axiom on $\mathbb{R}\times X$. The converse is obvious. 

Proof of statement (2) is similar. 

Finally, we show statement (3). Let $S$ have domain $\{0\}$ and $s \in \mathbb{R}.$ By shift invariance,
$\sigma(s, \phi) \in S$. Therefore, $0 \in \mathrm{dom}\, \sigma(s, \phi)$. However, this is equivalent to $s \in \mathrm{dom}\, \phi.$ Therefore, $S$ has domain $\mathbb{R}$.
\end{proof}

In particular, we do not need to consider partial maps as long as we are concerned with shift-invariant sets having domain $\{0\}$.

\begin{corollary}
Let $S \subset C_{s}(\mathbb{R}, X)$ be shift-invariant and have domain $\{0\}$.
Then, $S \subset C(\mathbb{R}, X)$.\end{corollary}
From the results above, for a shift-invariant subset of $C_{s}(\mathbb{R}, X)$, we can say that $S$ satisfies the existence axiom on $X$ if $S$ satisfies the existence axiom on $\{0\}\times X.$ In what follows, we use this abbreviation for shift-invariant subsets of $C_{s}(\mathbb{R}, X)$.

Flows can be identified with well-behaved subspaces of $C(\mathbb{R},X).$

\begin{definition}
A subset $S \subset C(\mathbb{R}, X)$ is a \textbf{well-posed set} if $S$ is a shift-invariant set that satisfies the compactness, uniqueness, and existence axioms on $X$. \end{definition}

\begin{theorem}[Yorke]\label{yorke}
If $X$ is locally compact and $S$ is a well-posed set, then $(\sigma, S)$ and $(\Phi, X) $ are topologically conjugate, where the flow $\Phi: \mathbb{R}\times X \to X$ is defined as
\[
	\Phi(t, x) := \phi(t),
\]
where $\phi \in S$ satisfies $\phi(0) = x.$ Conversely, if $\Phi$ is a continuous flow, then set $S =\{\phi(-):=\Phi(-, x) \mid x \in X \}$ is a well-posed set.
\end{theorem}
In Yorke's formalism, concatenating trajectories is not always possible, and this condition is formulated as an additional axiom.
\begin{definition}
A subset $S \subset C_{s}(\mathbb{R}, X)$ satisfies the \textbf{switching axiom} if $S$ contains map $\psi$ defined by \[
	\psi(t) = \begin{cases}
				\phi_1(t) &(t \leq \tau)\\
				\phi_2(t) &(t \geq \tau)
			\end{cases}
\]
whenever $\phi_1, \phi_2 \in S$ satisfy
\[
	\phi_1(\tau) =\phi_2(\tau)
\]
for some $\tau \in \mathrm{dom}\, \phi_1 \cap \mathrm{dom}\, \phi_2.$
\end{definition}
\begin{example}
For an ODE of the form $\dot{x} = f(x)$, where $f:\mathbb{R}^n \to \mathbb{R}^n$ is globally Lipschitz, the set of all solutions is a well-posed set. For a differential inclusion $\dot x \in F(x)$, the solution set satisfies the compactness axiom and existence axiom if $F$ is sufficiently well-behaved. This is a consequence of the compactness of the solution set (see Corollaries 4.4 and 4.5 in \cite{ smirnov2022introduction}).
\end{example}
\begin{example}\label{ex_fil}
Let us consider a Filippov system (for details, see \cite{filippov}) on $\mathbb{R}^2$ defined by a piecewise smooth vector field
\[
	Z(x,y) = 
\begin{cases}
	(0,-1) & y > 0\\
	(0,0) & y = 0\\
	(0,1) & y <0
\end{cases}
\]
The set of all solutions can be shown to satisfy the compactness axiom using Theorem \ref{cptness}.
\end{example}
\section{Topological properties of a solution space and the correspondence with dynamics}\label{solsp}
In this section, we consider how Yorke's axioms induce topological properties of a solution space and how the dynamics of the solution space, which are induced by the shift map, relate to the dynamics in the phase space.

\subsection{Topological properties of a solution space}
While Yorke's axioms describe the well-posedness of initial value problems, they have consequences for the topological properties of a solution space. These results are used in the following discussion of the dynamics.
\begin{theorem}\label{cl_thm}
If subset $S \subset C_{s}(\mathbb{R}, X)$ satisfies the compactness axiom, $S$ is closed.
\end{theorem}
\begin{proof}
We assume that $S$ is not closed. Then, we may take a sequence $\{\phi_n\}_n$ in $S$ such that $\phi_n$ converges to $\phi \not \in S.$ For a fixed $t \in \mathrm{dom} \phi,$ we have $t\in \mathrm{dom} \phi_n$ for sufficiently large $n$ and $\phi_n(t) \to \phi(t)$ in $X.$ Therefore, the set $\overline{ \{(t, \phi_n(t))\}_n }\subset \mathbb{R} \times X$ is compact, and 
\[
	(t, \phi_n) \in \qty(e_S)^{-1} \left( \overline{ \{(t, \phi_n(t))\}_n}\right).
\]
By the compactness axiom, we have $\phi_{n_{i}} \to \psi$ for some subsequence and $\psi \in S.$ By construction, we have $t \in \mathrm{dom} \psi \cap \mathrm{dom} \phi$. Furthermore, we have $\phi_{n_i}(s) \to \psi(s)$ for all $s \in \mathrm{dom} \psi.$ As $X$ is Hausdorff, it follows that $\phi$ and $\psi$ coincide on $\mathrm{dom} \psi \cap \mathrm{dom} \phi \neq \emptyset.$ By the maximality of $\phi$ and $\psi,$ we have $\phi = \psi,$ which contradicts the assumption $\phi \not \in S.$
\end{proof}

For the compactness of the solution space, we have the following analog of Barbashin's theorem, which concerns the compactness of the set of solutions of a set-valued process (see, for example, Proposition 2.3 in \cite{benaim2000ergodic} and Theorem 3 in \cite{CARABALLO2003692}):
\begin{theorem}\label{barbashin}
Let X be compact and $S \subset C_{s}(\mathbb{R}, X)$ satisfy the compactness axiom and have domain $\{0\}.$ Then, $S$ is compact.
\end{theorem}
\begin{proof}
The space $\{0\}\times S = \qty(e_S)^{-1} \left(\{0\}\times X\right)$ is compact. Note that the projection of the map component $\mathrm{p}: \mathbb{R} \times C_{s}(\mathbb{R}, X) \to C_{s}(\mathbb{R}, X)$ is continuous. Because $S = \mathrm{p} \left(\{0\}\times S \right),$ $S$ is compact as it is a continuous image of a compact set.
\end{proof}
Conversely, the compactness of a solution space implies the compactness axiom.
\begin{theorem}\label{cpt_then_cptness}
If subset $S \subset C_{s}(\mathbb{R}, X)$ is compact, $S$ satisfies the compactness axiom.
\end{theorem}
\begin{proof}
Let $W \subset \mathbb{R}\times X$ be compact. To show that $ \qty(e_S)^{-1}W$ is compact, let us consider a sequence $\{(t_n, \phi_n)\}_n \subset \qty(e_S)^{-1}W.$ Because $\{\phi_n\}_n \subset S$ and $S$ is compact, we may take a convergent subsequence $\{\phi_{n_i}\}_i.$ As $\{\qty(t_{n_i}, \phi_{n_i}\qty(t_{n_i}))\}_i \subset W,$ there exists a convergent subsequence of $\{t_{n_i}\}_i$. Therefore, $\{(t_n, \phi_n)\}_n$ has a convergent subsequence.
\end{proof}
Combining Theorems \ref{barbashin} and \ref{cpt_then_cptness}, we obtain the following characterization of compactness when the phase space is compact.
\begin{corollary}\label{cor_cpt}
If $X$ is a compact metric space, subset $S \subset C(X, \mathbb{R})$ is compact if and only if the compactness axiom is satisfied.
\end{corollary}

\subsection{Correspondence between solution space and dynamics}
Now, we consider how the dynamics in the solution space, which are induced by the shift map, correspond to those in the phase space.

First, we consider the basic dynamical properties.
In the axiomatic theory of ODE, we have a result similar to that of the Bebutov dynamical system (Theorem 1.20 and 2.20 in \cite{sibirskij1975introduction}). 
\begin{theorem}\label{bebutov}
Let $S$ be a well-posed set and $\Phi$ be the flow induced by $S$ via Theorem \ref{yorke}. Then, 
\begin{enumerate}
\item $\phi \in S$ is constant if and only if $\phi(0)$ is an equilibrium point of $\Phi$.
\item $\phi \in S$ is periodic if and only if $\phi(0)$ is a periodic point of $\Phi$.
\end{enumerate}
\end{theorem}

Therefore, we call $x \in X$ an \textbf{equilibrium point} of a shift-invariant set $S$ if there is a constant map $\phi \in S$ with $\phi(0) = x.$ Note that constant or periodic maps in $C_s(\mathbb{R}, X)$ have domain $\mathbb{R}$.
\begin{theorem}
Let $S\subset C_{s}(\mathbb{R}, X)$ be a shift-invariant closed set. Then, the set of equilibrium points of $S$ is closed.
\end{theorem}
\begin{proof}
Let $\{x_n\}_{n}$ be a sequence of equilibrium points of $S$ converging to $x.$ If we let $\phi$ be a constant map with $\phi(0) = x,$ it is clear that $\phi_n \to \phi,$ where $\phi_n$ are constant maps with $\phi_n(0) = x_n.$ Therefore, $\phi \in S,$ and $x$ is an equilibrium point of $S$.
\end{proof}

Now, we consider general types of invariant sets.
There are two types of invariance for dynamical systems without uniqueness of trajectories. In the context of the axiomatic theory of ODE, these are formulated as follows:
\begin{definition}[Weak and strong invariance]
Let $S\subset C_{s}(\mathbb{R}, X)$ be a shift-invariant set. For a subset $A \subset X,$
\begin{enumerate}
	\item $A$ is \textbf{weakly invariant} if, for all $x \in A,$ there exists $\phi \in S$ such that $\phi(0) = x$ and $\mathrm{im}\, \phi \subset A.$
	\item $A$ is \textbf{strongly invariant} if $\mathrm{im}\, \phi \subset A$ for all $\phi \in S$ with $\phi(0) \in A.$

\end{enumerate}
\end{definition}
\begin{remark}
By definition, strong invariance implies weak invariance. If $A$ is weakly invariant, the existence axiom is satisfied on $\{0\}\times A.$ If the uniqueness axiom is satisfied, these notions coincide.
\end{remark}

Although we defined the invariance in terms of the evaluation at $0$, this is not essential, as is observable in the next result.
\begin{lemma}\label{lem_inv}
Let $S\subset C_{s}(\mathbb{R}, X)$ be a shift-invariant set. For a subset $A \subset X,$
\begin{enumerate}
	\item $A$ is weakly invariant if and only if, for all $x \in A,$ there exists $\phi \in S$ such that $x \in \mathrm{im}\, \phi \subset A.$
	\item $A$ is strongly invariant if and only if $\mathrm{im}\, \phi \cap A \neq \emptyset$ implies $\mathrm{im}\, \phi \subset A$ for all $\phi \in S$.

\end{enumerate}
\end{lemma}
\begin{proof}
(1): The ``only if'' part is trivial. Let $A \subset X$ be such that for all $x \in A,$ there exists $\phi \in S$ with $x \in \mathrm{im}\, \phi \subset A.$ For each $x \in A$, we take $\phi \in S$ in the assumption and let $t \in \mathrm{dom}\, \phi$ such that $x = \phi(t).$ By the shift-invariance of $S$, we have $\psi := \sigma(t, \phi) \in S.$ The map $\psi$ satisfies $\psi
(0) = x$ and $\mathrm{im}\, \psi \subset A $.

(2): The ``if'' part is trivial. Let $A$ be strongly invariant and $\mathrm{im}\, \phi \cap A \neq \emptyset$ for a map $\phi \in S.$ By considering $\psi := \sigma(t, \phi)$, where $\psi(0) = \phi(t) \in \mathrm{im}\, \phi \cap A,$ we obtain $\mathrm{im}\, \psi \subset A$, which is equivalent to $\mathrm{im}\, \phi \subset A.$
\end{proof}
Weak invariance can be used to establish the existence of globally defined solutions.
\begin{theorem}\label{thm_glob}
Let $S\subset C_{s}(\mathbb{R}, X)$ be a shift-invariant set satisfying the compactness axiom and $A \subset X$ be a weakly invariant subset. If $\bar A$ is compact, there exists a shift-invariant subset $S' \subset S \cap C\qty(\mathbb{R}, X) $, such that $\phi(0) \in A$ for all $\phi \in S'.$
\end{theorem}
\begin{proof}
For each $x\in A,$ there exists $\phi_x \in S$ with $\phi_x(0) =x$ and $\mathrm{im}\, \phi_x \subset A.$ From Theorem 2.17 in \cite{suda2022equivalence}, we have $\mathrm{dom}\, \phi_x = \mathbb{R}.$

Now, we consider the set
\[
	S' := \bigcup_{t \in \mathbb{R}} \{\sigma\qty(t,\phi_x) \mid x \in A\}.
\]
By definition, $S' $ is a shift-invariant subset of $S$ and $\phi(0) \in A$ for all $\phi \in S'.$ 
\end{proof}

The notion of invariance can be described in terms of the corresponding sets in the solution space. In particular, for a shift-invariant $S\subset C_s(\mathbb{R}, X)$ where $S$ satisfies the compactness axiom, we can establish a one-to-one correspondence between a class of compact invariant subsets of $S$ and compact weakly invariant subsets of $X$.

To state this result, we introduce several concepts and results. We define the evaluation-at-$0$ map \[\pi: C_s(\mathbb{R},X) \to X\] by $\pi(\phi) = \phi(0).$ By an argument similar to Lemma \ref{lem_eval}, we can show that $\pi$ is a partial map with open domain. Alternatively, we can verify this assertion by observing that $\pi$ is a composition of an injection $C_s(\mathbb{R},X) \to \mathbb{R} \times C_s(\mathbb{R},X),$ $e$ and a projection $\mathbb{R}\times X \to X.$
 
\begin{definition}
Let $S \subset C_s(\mathbb{R},X)$ and $A \subset X$. The \textbf{$A$-section of $S$} is defined by
\[
	S_A := \pi^{-1}(A) \cap S.
\] 
\end{definition}
We can say that $A$-section is the set in the space $S$ which corresponds to $A,$ and helps describe the dynamics in the phase space. 
\begin{remark}
By definition, $\phi \in S_A$ implies that $0 \in \mathrm{dom}\, \phi$. This observation has the following consequence.
In general, invariant subset $S$ of $C_s(\mathbb{R},X)$ may consist only of partial maps. However, if $S_A$ contains an invariant set $S',$ then $S' \subset C(\mathbb{R}, X).$
\end{remark}
We say an invariant subset $S' \subset S$ is \textbf{maximal with respect to phase space} if $\pi(S') = \pi(S'')$ implies $S'' \subset S'$ for all invariant subset $S'' \subset S$. If an invariant subset is maximal with respect to phase space, it is the largest among the invariant subsets occupying the same phase space.

\begin{theorem}[Main Theorem A]\label{thm_cor}
Let $S\subset C_s(\mathbb{R}, X)$ be a shift-invariant set that satisfies the compactness axiom. 
Then, there exists a one-to-one correspondence between the sets
\[
  \{S' \subset S \mid \text{shift invariant, maximal w.r.t. phase space and }\pi(S')\text{ is compact} \}
\]
and
\[
	\{A\subset X\mid A \text{ is weakly invariant and compact} \}.
\]
\end{theorem}
To prove this theorem, we prepare several lemmas.
\begin{lemma}\label{lem_winv}
Let $S\subset C_{s}(\mathbb{R}, X)$ be a shift-invariant set. If $S' \subset S$ is shift-invariant, $\pi(S')$ is weakly invariant with respect to $S.$ 
\end{lemma}
\begin{proof}
Let $x \in \pi(S').$ Then, there exists $\phi \in S' \subset S$ with $\phi(0) = x.$ Because $S'$ is shift-invariant, we have
\[
	\phi(t) = \pi\qty(\sigma(t, \phi)) \in \pi\qty(\sigma(t, S')) = \pi(S')
\] 
for all $ t \in \mathrm{dom}\, \phi.$ Therefore, $\pi(S')$ is weakly invariant with respect to $S$.
\end{proof}

If $S$ and $S' \subset S$ are shift invariant, each $A \subset S$ with $S' \subset A$ has the maximal shift invariant subset. We use this property in the next lemma.

\begin{lemma}\label{lem_inv1}
Let $S\subset C_s(\mathbb{R}, X)$ be a shift-invariant set that satisfies the compactness axiom. A compact subset $A \subset X$ is weakly invariant with respect to $S$ if and only if $\pi\qty(\mathrm{inv}\,S_A) = A,$ where $\mathrm{inv}\,S_A$ is the maximal shift invariant subset of $S_A.$
\end{lemma}
\begin{proof}
Let $A \subset X$ be weakly invariant with respect to $S$. By Theorem \ref{thm_glob}, there exists a shift-invariant subset $S' \subset S \cap C(\mathbb{R}, X)$, such that $\pi(S') = A,$ which implies $S' \subset S_A.$ Therefore, we have $S' \subset \mathrm{inv}\, S_A$ and $A = \pi(S') \subset \pi\qty(\mathrm{inv}\,S_A).$ Conversely, we have $\pi\qty(\mathrm{inv}\,S_A) \subset \pi\qty(S_A) \subset A$ by definition. Therefore, we obtain $A = \pi\qty(\mathrm{inv}\,S_A).$

The converse follows from Lemma \ref{lem_winv}.
\end{proof}
\begin{lemma}\label{lem_inv2}
Let $S\subset C_s(\mathbb{R}, X)$ be a shift-invariant set that satisfies the compactness axiom and a compact subset $A \subset X$ be weakly invariant with respect to $S.$ Then, the maximal shift invariant subset of $S_A$ is maximal with respect to phase space.
\end{lemma}
\begin{proof}
Let $S' \subset S$ be a shift-invariant subset with $\pi(S') = A.$ Then, we have $S' \subset \pi^{-1}(A)$, which implies that $S' \subset \mathrm{inv}\,S_A.$
\end{proof}
\begin{proof}[Proof of Theorem \ref{thm_cor}]
Let $\mathrm{Inv}_c(S)$ be the set of all $S' \subset S$ such that shift invariant, maximal with respect to phase space and $\pi(S')$ is compact. Then, we can define a map 
\[
	\mathrm{Inv}_c(S) \to \{A\subset X\mid A \text{ is weakly invariant and compact} \}
\]
by $S' \mapsto \pi(S')$ by Lemma \ref{lem_winv}. This map is surjective by Lemmas \ref{lem_inv1} and \ref{lem_inv2}. By the definition of maximality with respect to the phase space, this is also injective.
\end{proof}

In fact, if a weak invariant set $A$ is compact, $\mathrm{inv}\,S_A$ is also compact. This is established by the following result and the fact that the closure of an invariant set is also invariant. Therefore, Theorem \ref{thm_cor} can be considered as a correspondence between a class of compact invariant subsets of $S$ and compact weakly invariant subsets.
\begin{lemma}\label{thm_cptness}
Let $S$ satisfy the compactness axiom and the existence axiom on $\{0\}\times A$. Then, $S_A$ is compact if $A$ is compact.
\end{lemma}
\begin{proof}
As $e_{S}^{-1}\qty(\{0\} \times A) = \{0\} \times S_A$ and $\{0\}\times A$ is compact, $S_A$ is compact.
\end{proof}
\begin{example}
For a shift-invariant set $S\subset C_s(\mathbb{R}, X),$ equilibrium points correspond with constant maps in $S$ according to Theorem \ref{bebutov}. Theorem \ref{thm_cor} can be viewed as a generalization of this correspondence by the following observation: Obviously, the maximal shift-invariant subset of $S_{\{x\}}$ comprises at most one map, that is, the constant map with value $x$. Conversely, a singleton $\{x\}$ is weakly invariant if and only if $x$ is an equilibrium point of $S$.
\end{example}
If the phase space $X$ is compact and the solutions are defined globally, Theorem \ref{thm_cor} can be improved.
\begin{theorem}
Let $X$ be compact and $S\subset C(\mathbb{R}, X)$ be a shift-invariant set. 
There exists a one-to-one correspondence between the sets
\[
  \{S' \subset S \mid \text{shift invariant and maximal w.r.t. phase space}\}
\]
and
\[
	\{A\subset X\mid A \text{ is weakly invariant} \}.
\]
\end{theorem}
\begin{proof}
Let $\mathrm{Inv}_c(S)$ be the set of all $S' \subset S$, such that the shift is invariant and maximal with respect to the phase space. Using Lemma \ref{lem_winv}, a map
\[
	\mathrm{Inv}_c(S) \to \{A\subset X\mid A \text{ is weakly invariant} \}
\]
can be defined by $S' \mapsto \pi(S')$. If $A \subset X$ is weakly invariant, we can show the existence of a shift-invariant subset $S' \subset S_A$ by an argument similar to Theorem \ref{thm_glob}, enabling us to obtain results similar to those of Lemmas \ref{lem_inv1} and \ref{lem_inv2}.
\end{proof}

For strong invariance, we have the following result.
\begin{theorem}
Let $S\subset C(\mathbb{R}, X)$ be a shift-invariant set that satisfies the compactness axiom. A subset $A \subset X$ is strongly invariant with respect to $S$ if and only if $S_A$ is invariant; that is, $S_A = \mathrm{inv}\,S_A$.
\end{theorem}
\begin{proof}
Let $A \subset X$ be strongly invariant and $\phi \in S_A.$ From Lemma \ref{lem_inv}, we have $\mathrm{im}\, \phi \subset A$ because $\phi(0) \in \mathrm{im}\, \phi \cap A.$ Therefore, $\sigma(t, \phi) \in \pi^{-1}(A)$ for all $t \in \mathbb{R}$, which implies that $S_A$ is shift-invariant. 

Conversely, if $S_A$ is shift-invariant, all $\phi \in S$ with $\phi(0) \in A$ satisfy $\sigma(t, \phi ) \in S_A$. Therefore, $\phi(t) \in A$ for all $t \in \mathbb{R}.$
\end{proof}

\section{Asymptotic properties}\label{asym}
In this section, we consider the asymptotic properties of the dynamics of the phase space induced by the solution space. In particular, we discuss how the definition of limit sets can be generalized.

First, we recall the definition of the limit sets for a flow.
\begin{definition}
For a flow $\Phi$ on $X$ and $A \subset X,$ we define 
\[
\begin{aligned}
	\omega(A) &:= \bigcap_{t>0} \overline{\bigcup_{s \geq t} \Phi(s, A)}\\
	\alpha(A) &:= \bigcap_{t<0} \overline{\bigcup_{s \leq t} \Phi(s, A)}.
\end{aligned}
\]
\end{definition}
\begin{remark}
For a flow $\Phi$ on $X$ and $A \subset X$, where $X$ is the first countable space, we have that
\[
\begin{aligned}
	\omega(A) &= \{y \in X \mid \text{there exist } a_n \in A,t_n \text{ s.t. } \Phi(t_n, a_n ) \to y, t_n \to \infty\},\\
	\alpha(A) &= \{y \in X \mid \text{there exist } a_n \in A,t_n \text{ s.t. } \Phi(t_n, a_n ) \to y, t_n \to -\infty \}.
\end{aligned}
\]
Additionally, we note that $\omega(A)$ and $\alpha(A)$ are closed-invariant sets.
\end{remark}
Because $(\sigma, S)$ is a flow, the notion of limit sets is naturally defined. If $S$ is a well-posed set, the limit sets of $(\sigma, S)$ can be described in terms of sections.
\begin{theorem}
Let $X$ be locally compact, $S$ be a well-posed set, and $\Phi$ be the flow induced by $S$ via Theorem \ref{yorke}. Then,
\begin{enumerate}
	\item For each $x \in X,$ $S_{\omega(x)} = \omega(S_x).$
	\item For each $x \in X,$ $S_{\alpha(x)} = \alpha(S_x).$
\end{enumerate}
\end{theorem}
\begin{proof}
We show the result for the $\omega$ limit set. Note that we have $S_x = \{\phi\}$, where $\phi(t) = \Phi(t, x).$ 

If $\psi \in S_{\omega(x)},$ we have $\psi(0) = \lim_{n\to \infty} \Phi(t_n, x)$ for some sequence $t_n \to \infty$ as $n \to \infty.$ As the closure $W$ of the set $\{\qty(0, \Phi(t_n, x)) \mid n \in \mathbb{N}\}$ is compact, we may find a convergent subsequence of $\qty(0, \sigma(t_n, \phi)) \in e_{S}^{-^1} W.$ Combined with the uniqueness axiom, it shows that $\psi \in \omega(S_{x}).$

Conversely, let $\psi \in \omega(S_{x}).$ Then, we have $\psi = \lim_{n\to \infty} \sigma(t_n, \phi)$ for some sequence $t_n \to \infty$ as $n \to \infty.$ By the continuity of evaluation at $0$, we obtain $\psi(0) = \lim_{n\to \infty} \phi(t_n),$ which implies that $\psi(0) \in \omega(x).$
\end{proof}
The preceding theorem implies that $\pi\left( \omega(S_x)\right) = \omega(x)$ for a flow. This observation motivates the following definition:
\begin{definition}
Let $S$ be a shift-invariant set that satisfies the compactness axiom. For each $x \in X,$ we define
\[
\begin{aligned}
	\omega_S(x) &:= \pi\left( \omega(S_x)\right),\\
	\alpha_S(x) &:= \pi\left( \alpha(S_x)\right).
\end{aligned}
\]
\end{definition}

We consider how this definition of limit sets relates to the notion of limit sets used in Yorke's original paper. 
\begin{definition}[Strauss-Yorke definition of limit set, \cite{strauss1967asymptotically, yorke1969spaces}]
Let $S$ be a shift-invariant set that satisfies the compactness axiom. For $B \subset S,$ we define
\[
\begin{aligned}
	\Lambda^+(B)&:= \{y \in X \mid \text{there exist } \phi_n \in B,t_n \in \mathrm{dom}\, \phi_n \text{ s.t. } \phi_n(t_n) \to y, t_n \to \infty\}\\
	\Lambda^-(B)&:= \{y \in X \mid \text{there exist } \phi_n \in B,t_n \in \mathrm{dom}\, \phi_n \text{ s.t. } \phi_n(t_n) \to y, t_n \to -\infty \}
\end{aligned}
\]
\end{definition}

\begin{lemma}
Let $S$ be a shift-invariant set, where the compactness axiom is satisfied. For $B \subset S,$ we have
\[
\begin{aligned}
	\Lambda^+(B)&= \pi(\omega(B))\\
	\Lambda^-(B)&= \pi(\alpha(B)).
\end{aligned}
\]
In particular, we have 
\[
\begin{aligned}
	\omega_S(x) &= \Lambda^+(S_x),\\
	\alpha_S(x) &= \Lambda^-(S_x).
\end{aligned}
\]
\end{lemma}
\begin{proof}
Let $x \in \Lambda^+(B).$ Then, we may find $\phi_n \in B$ and $t_n \in \mathrm{dom}\, \phi_n $ such that $ \phi_n(t_n) \to x$ and $t_n \to \infty.$ As $\phi_n \in B \subset S,$ we have $\psi_n := \sigma(t_n, \phi_n) \in S$ owing to the invariance. Therefore, we obtain
\[
	(0, \psi_n) \in e_S^{-1}\overline{\{(0, \phi_n(t_n))\}_n}.
\]
From the compactness axiom, we may find a convergent subsequence $\{\psi_{n_i}\}_i.$ For the limit $\psi,$ we see that $\psi \in \omega(B).$ By the continuity of $\pi,$ we have
\[
	x=\lim_{i \to \infty} \phi_{n_i}(t_{n_i}) = \lim_{i \to \infty} \pi(\psi_{n_i}) =\pi(\psi).
\] 
Therefore, $x\in \pi(\omega(B)).$

Conversely, let $x\in \pi(\omega(B)).$ Then, we may find $\phi_n \in B$ and $t_n$ such that
\[
	\psi = \lim_{n \to \infty} \sigma(t_n, \phi_n),
\]
$t_n \to \infty$ and $\pi(\psi)=x.$ As $0 \in \mathrm{dom}\, \psi,$ we have $t_n \in \mathrm{dom}\, \phi_n$ for sufficiently large $n$, and hence, $ \phi_n(t_n) \to x.$ 
\end{proof}
Now, we generalize classical results on the limit sets of a flow. 
\begin{theorem}
Let $S$ be a shift-invariant set that satisfies the compactness axiom. If a compact subset $K \subset X$ is weakly invariant, $\omega_S(x) \neq \emptyset$ for all $x \in K.$
\end{theorem}
\begin{proof}
Let $x \in K.$ Because $K$ is weakly invariant, we may apply Theorem \ref{thm_glob} to obtain $\emptyset \neq \Lambda^+(S_x) = \omega_S(x).$
\end{proof}

\begin{theorem}[Main Theorem B]
Let $S$ be a shift-invariant set that satisfies the compactness axiom. Then, the limit sets $\omega_S(x)$ and $\alpha_S(x)$ are closed and weakly invariant for all $x \in X$.
\end{theorem}
\begin{proof}
If $\omega_S(x)$ is empty, the statements hold vacuously. We assume that $\omega_S(x) \neq \emptyset.$ By the invariance of $\omega(S_x)$ and Lemma \ref{lem_winv}, $\omega_S(x)$ is weakly invariant.

We show that $\omega_S(x)$ is closed. Let $y_n \in \omega_S(x)$ with $y = \lim_{n \to \infty} y_n$; then, we have $y_n = \phi_n(0)$ for $\phi_n \in \omega(S_x) \subset S.$ As $S$ satisfies the compactness axiom, we can find a convergent subsequence $\phi_{n_i} \to \psi$ by considering $\{\qty(0, y_n)\}_n.$ Because $\omega(S_x)$ is closed, we have $\psi \in\omega(S_x).$ By the continuity of $\pi,$ we have $y = \psi(0) \in \pi(\omega(S_x)).$ 
\end{proof}

The notion of recurrence, which is used in Section \ref{inv_meas}, can be defined as follows:
\begin{definition}
Let $S \subset C_s(\mathbb{R}, X)$ be compact and shift invariant.
A point $x \in X$ is \textbf{recurrent} for $S$ if 
\[
	x \in \omega_S(x).
\]
\end{definition}
Finally, let us briefly comment on the notion of attractor. As discussed in the next section, Yorke's theory is compatible with other formalisms of generalized dynamical systems. Accordingly, we can introduce various definitions of attractors.
\section{Comparison with other formalisms}\label{comp}
In this section, we compare the axiomatic theory of ODE with other formalisms of generalized dynamical systems, namely, generalized semiflow and multivalued semigroup.

 For simplicity, we consider the simplest case for the Yorke formalism, that is, when the phase space $X$ is a compact metric space and each map is globally defined.
\begin{remark}
In this case, the topology of a solution space in Yorke's theory coincides with that of Filippov's because both are reduced to the compact-open topology on $C(\mathbb{R}, X).$ 
\end{remark}
The generalized semiflow introduced by Ball is an important and frequently used formalism.
\begin{definition}[Generalized semiflow \cite{ball1997continuity}]\label{def_gen}
A \textbf{generalized semiflow} $G$ on $X$ is a family of maps $\phi: \mathbb{R}_{\geq 0} \to X$ that satisfy the following:
\begin{enumerate}
	\item For each $z \in X,$ there exists at least one $\phi \in G$ with $\phi(0) = z.$
	\item If $\phi \in G$ and $\tau \geq 0,$ then $\phi^{\tau} \in G$, where $\phi^{\tau}(t) = \phi(t+\tau).$
	\item If $\phi, \psi \in G$ and $\phi(t) = \psi(0),$ then $\theta \in G,$ where
	\[
		\theta(s) = \begin{cases}
						\phi(s) & s \leq t\\
						\psi(s-t)& s \geq t
					\end{cases}
	\]
		\item If $\phi_j \in G$ with $\phi_j(0) \to z,$ there exists a subsequence $\{\phi_{\mu}\}$ of $\{\phi_j\}$ and $\phi \in G$ with $\phi(0) = z$ and $\phi_{\mu} \to \phi$ pointwise.

\end{enumerate}
\end{definition}
As observed, the definition of generalized semiflow bears a similarity to Yorke's theory. The next result states this connection in an exact form:
\begin{theorem}\label{thm_y2g}
If a shift-invariant set $S$ satisfies the compactness, existence, and switching axioms on $\mathbb{R}_{\geq 0}\times X$, $S$ is a generalized semi-flow (if regarded as a set of maps of the form $\mathbb{R}_{\geq 0} \to X$).
\end{theorem}
\begin{proof}
We verify conditions for the generalized semilflow individually.

Conditions (1) and (2) follow from the existence axiom and shift-invariance, respectively. 

For condition (3), let $\phi, \psi \in S$ satisfy $\phi(t) = \psi(0)$ for some $t > 0.$ We then apply the switching axiom to $\phi$ and $\sigma(-t, \psi)$ to obtain their concatenation $\theta$.

To show (4), let $\phi_j \in S$ with $\phi_j(0) \to z.$ By applying the compactness axiom to the closure of $\{\qty(0, \phi_j(0))\}_j,$ we obtain a convergent subsequence of $\{\phi_j\},$ which is also pointwise convergent. 
\end{proof}
This result provides a way to obtain a generalized semiflow from a well-behaved shift-invariant subset of $C(\mathbb{R},X)$. Later, we consider the converse problem of obtaining a well-behaved shift-invariant subset of $C(\mathbb{R},X)$ from a generalized semiflow. 

In the context of set-valued dynamical systems, the multivalued semigroup is often used. Here, we introduce this concept in terms of a generalized semiflow.
\begin{definition}[\cite{caraballo2003comparison, moussa2017invariant}]
A \textbf{multivalued semigroup determined by a generalized semiflow $G$} is a family of set-valued operators $\{V(t)\}_{t \in \mathbb{R}_{\geq 0}}$ on $X$ given by
\[
	V(t) E := \{\phi(t) \mid \phi \in G \text{ and } \phi(0) \in E\}
\]
for $E \subset X.$
\end{definition}
For a multivalued semigroup determined by a generalized semiflow, we have
\[
	V(t + s) E = V(t) V(s) E
\]
if $t, s \geq 0$ and $E \subset X.$

If $V$ is a multivalued semigroup determined by a generalized semiflow $G \subset C(\mathbb{R}, X),$ we define
\[
	V(-t) E := \{\phi(-t) \mid \phi \in G \text{ and } \phi(0) \in E\}
\]
for $t>0.$

\begin{lemma}\label{lem_smg}
Let a multivalued semigroup $V$ be determined by a generalized semiflow $G \subset C(\mathbb{R}, X),$ where $G$ is shift invariant. Then the following hold:
\begin{enumerate}
	\item For all $E \subset X$, we have
	\[
		E \subset V(t) V(-t) E	
	\]
for all $t \in \mathbb{R}$.  
	\item For all $E \subset X$, we have
\[
	V(t)^{-1} E := \{x \in X \mid V(t) x \cap E \neq \emptyset \} = V(-t) E
\]
for all $t \in \mathbb{R}$. 
\end{enumerate}
\end{lemma}
\begin{proof}
(1): Let $x \in E.$ By condition (1), there exists $\phi \in G$ with $\phi(0) = x.$ From the shift invariance of $G,$ we have $\sigma(-t, \phi) \in G.$ Then, we have $\phi(-t) \in V(-t) E$; therefore, $x = \sigma(-t, \phi)(t) \in V(t)V(-t) E. $
 
(2): Let $x \in V(t)^{-1} E$. By definition, there exists $\phi \in G$ with $\phi(0) = x $ and $\phi(t) \in E.$ By shift invariance, we have $\sigma(t,\phi)\in G.$ Therefore, $x = \sigma(t,\phi)(-t)\in V(-t) E.$ Conversely, let $x \in V(-t) E.$ Then, there exists $\phi \in G$ with $\phi(0) \in E $ and $\phi(-t) =x.$ By considering $\sigma(-t, \phi)$, we obtain $x \in V(t)^{-1} E$.
\end{proof} 
\begin{remark}
In general, inclusion in the first assertion is strict.
\end{remark}
Multivalued semigroup can be described in terms of a section.
\begin{lemma}\label{lem_ss}
Let a multivalued semigroup $V$ be determined by a generalized semiflow $G \subset C(\mathbb{R}, X),$ where $G$ is shift invariant. Then, for all $A \subset X$ and $t\in \mathbb{R}$, we have
\[
	V(t) A = \pi\qty(\sigma(t, S_A)).
\]
\end{lemma}
\begin{proof}
Let $x \in V(t) A.$ Then, there exists $\phi \in S$ with $\phi(t) = x$ and $\phi(0) \in A.$ Therefore, $\phi \in S_A$ and $x = \sigma(t, \phi)(0) \in \pi\qty(\sigma(t, S_A)).$ Conversely, if $x \in \pi\qty(\sigma(t, S_A)),$ there exists $\phi \in \sigma(t, S_A)$ with $\phi(0) = x.$ Because $\sigma(-t, \phi)(0) \in A$, it follows that $x = \sigma(-t, \phi)(t) \in V(t) A.$
\end{proof}
The next notion is essential for describing the relationship between the axiomatic theory of ODE and multivalued semigroup.
\begin{definition}
A map $\xi: \mathbb{R} \to X$ is called a \textbf{complete trajectory} if $\xi(t+s) \in V(t)\xi(s)$ for all $t \geq 0$ and $s \in \mathbb{R}.$
\end{definition}
We note that the complete trajectories can be concatenated:
\begin{lemma}\label{lem_swi_traj}
Let $\xi_1$ and $\xi_2$ be complete trajectories of a multivalued semigroup $V$ determined by a generalized semiflow. If $\xi_1(\tau) = \xi_2(\tau)$ for some $\tau \in \mathbb{R},$ then map $\xi$ defined by
\[
	\xi(t) = \begin{cases}
				\xi_1(t) &(t \leq \tau)\\
				\xi_2(t) &(t \geq \tau)
			\end{cases}
\]
is a complete trajectory.
\end{lemma}
\begin{proof}
The only problematic situation is when $t>0$ and $s < \tau$ satisfy $s+t \geq \tau$. In this case, we have
\[
	\begin{aligned}
		\xi_2(s+t) &\in V(s+t-\tau) \xi_2(\tau)\\
				&=V(s+t-\tau) \xi_1(\tau)\\
				&\subset V(s+t-\tau) V(\tau - s)\xi_1(s)\\
				&=V(t)\xi_1(s),
	\end{aligned}
\]
which implies $\xi(s+t) \in V(t) \xi(s).$
\end{proof}
The next lemma follows immediately from Lemma 5 in \cite{caraballo2003comparison}:
\begin{lemma}\label{lem_fwd}
Let $V(t)$ be a multivalued semigroup determined by a generalized semiflow $G \subset C(\mathbb{R}_{\geq 0},X).$ If $\xi$ is a continuous complete trajectory of $V,$ there exists $\phi \in G$ with $\xi(t) = \phi(t)$ for all $t \geq 0$.
\end{lemma}

The next result can be used to apply the axiomatic theory of ODE to the set of all continuous complete trajectories.
\begin{theorem}[Main Theorem C]\label{thm_g2y}
Let $V(t)$ be a multivalued semigroup determined by a generalized semiflow $G \subset C(\mathbb{R}_{\geq 0},X)$ on a compact metric space $X$.
Then, the set $\mathcal{K}$ of all continuous complete trajectories is nonempty and shift-invariant. As a solution space on $\pi\qty(\mathcal{K})$, $\mathcal{K}$ satisfies existence, compactness, and switching axioms.
\end{theorem}
\begin{proof}
First, we show that $\mathcal{K}$ is nonempty. For a point $x_0 \in X,$ there exists $\phi \in G$ with 
$\phi(0) = x_0.$ The set
\[ 
	\omega(\phi) := \bigcap_{t>0} \overline{\{\phi(s) \mid s > t\}}
\]
is nonempty because it is an intersection of compact sets. We fix a point $y_0 \in \omega(\phi).$ By the definition of $\omega(\phi),$ we construct a sequence $t_n >0$ with $t_n \to \infty$ and $\phi(t_n) \to y_0$ as $n \to \infty$.
From condition (4) in Definition \ref{def_gen}, there is a subsequence of $\{\phi^{t_n}\}$ that converges pointwise to $\psi \in G$ with $\psi(0) = y_0$. We assume that $\{\phi^{t_n}\}$ converges to $\psi$ pointwise.

Now, we extend $\psi$ continuously onto $[-1,\infty)$ so that the extension $\psi_1$ satisfies
\[
	\psi_1(s+t) \in V(t) \psi_1(s)
\]
for all $s \in [-1,\infty)$ and $t \geq 0.$ For sufficiently large $n$, we have $t_n -1 > 0.$ Therefore, a sequence $\{\phi(t_n -1)\}_n$ can be defined and has a convergent subsequence $\{\phi(t_{n_i} -1)\}_i$ because $X$ is compact. If we set $\phi(t_{n_i} -1) \to y_1$ as $i \to \infty,$ we may find a sub-subsequence $\{\phi^{t_{n_{i_j}}-1}\}_j$ with $\phi^{t_{n_{i_j}}-1} \to \tilde\psi_1$ as $i \to \infty,$ where $\tilde\psi_1 \in G$ and $\tilde\psi_1(0) = y_1$. For all $s \geq 1,$
\[
\begin{aligned}
	\tilde\psi_1 (s) &= \lim_{j \to \infty} \phi\qty((t_{n_{i_j}}-1) + s -1 + 1) \\
				&= \lim_{j \to \infty} \phi\qty(t_{n_{i_j}} + s -1 )\\
				&= \lim_{j \to \infty} \phi^{t_{n_{i_j}}} \qty(s -1 )\\
				&= \psi(s-1).
\end{aligned}
\]
Therefore, map $\psi_1(s) := \tilde\psi_1(s+1)$ has the required properties.

If we replace $t_n,$ $y_0$ and $\psi$ with $t_{n_{i_j}},$ $y_1$ and $\psi_1$, respectively, we can extend $\psi_1$ continuously onto $[-2,\infty)$. Similarly, we obtain a sequence of continuous maps $\psi_n :[-n, \infty)$ such that $\psi_{n+1} = \psi_n$ on $[-n,\infty).$ Then, a continuous map $\xi: \mathbb{R} \to X$ can be defined as
\[
	\xi(s) = \psi_n(s)
\]
if $s \geq -n.$ This $\xi$ provides a continuous complete trajectory of $V$.

The shift invariance of $\mathcal{K}$ immediately follows by definition. 

As noted in \cite{moussa2017invariant}, $\mathcal{K}$ is compact. Here we show this property for completeness. Let $\xi_n \in \mathcal{K}$ be the sequence. Because $X$ is compact, we assume that $\xi_n(-1)$ converges to point $x_1 \in X$ by renaming if necessary. By Lemma \ref{lem_fwd}, there is a subsequence $\{\xi_{n^1_i}\}_i$ and $\phi_1 \in G$ such that the restriction of $\sigma(-1,\xi_{n^1_i})$ to $[0,\infty)$ converges to $\phi_1$. This convergence is in the topology of compact-uniform convergence according to Theorem 2.3 in \cite{ball1997continuity}. If we apply the same argument to $\{\xi_{n^1_i}(-2)\}_i,$ we obtain a convergent subsequence $\{\xi_{n^2_i}\}_i$ and $\phi_2 \in G$ such that the restriction of $\sigma(-2,\xi_{n^2_i})$ to $[0,\infty)$ converges to $\phi_2$. For $-1 \leq s,$ we have
\[
	\phi_2(s+2) = \lim_{i \to \infty} \xi_{n^2_i}(s) = \lim_{i \to \infty} \xi_{n^2_i}(s+1 -1) = \phi_1(s+1).
\]
Similarly, for each $k \in\mathbb{N}$, we can construct subsequences $\{\xi_{n^k_i}\}_i$ and $\phi_k \in G$ such that the restriction of $\sigma(-k,\xi_{n^k_i})$ to $[0,\infty)$ converges to $\phi_k$ and $\phi_k(s+k) = \phi_m(s+m)$ for all $m \leq k$ and $s \in [-m,\infty).$ Then, a continuous complete trajectory $\xi: \mathbb{R} \to X$ can be defined by
\[
	\xi(s) = \phi_n(s+n)
\]
if $s \geq -n.$ From the compact uniform convergence of each $\{\sigma(-k,\xi_{n^k_i})\}_i$, we may find $i_k$ with
\[
	\sup_{s \in [-k, k]} d\qty(\xi_{n^k_{i_k}}(s), \xi(s)) < \frac{1}{k}
\]
for each $k \in \mathbb{N}$, showing that a subsequence of $\{\xi_n\}$ converges to a continuous complete trajectory in the topology of $C\qty(\mathbb{R}, X)$.

For $\mathcal{K}$, the existence axiom is obviously satisfied on $\pi\qty(\mathcal{K})$. Because of the compactness of $\mathcal{K}$, $\pi\qty(\mathcal{K})$ is compact. Therefore, it satisfies the compactness axiom by Corollary \ref{cor_cpt}. The switching axiom is satisfied by Lemma \ref{lem_swi_traj}. 
\end{proof}

By Theorem \ref{thm_y2g}, a sufficiently well-behaved shift-invariant subset of $C\qty(\mathbb{R}, X)$ can be regarded as a generalized semiflow, and hence, defines a multivalued semigroup. According to Theorem \ref{thm_g2y}, the set of all continuous complete trajectories satisfies the same axioms on $\pi(\mathcal{K})$. In fact, they are the same in this case.
\begin{theorem}
Let $X$ be a compact metric space.
If a shift-invariant set $S$ satisfies the compactness, existence, and switching axioms on $\mathbb{R}_{\geq 0}\times X$, $S$ is the set of all continuous complete trajectories of the multivalued semigroup determined by $S$.
\end{theorem}
\begin{proof}
Let $\mathcal{K}$ be the set of all the continuous complete trajectories of the multivalued semigroup determined by $S$.
By definition, $S \subset \mathcal{K}$. Conversely, let $\xi \in \mathcal{K}.$ For each $n \in \mathbb{N},$ we can find $\phi_n \in S$ such that it coincides with $\sigma(-n,\xi)$ on $[0,\infty)$ by Lemma \ref{lem_fwd}. For $s \in [-n, \infty)$, we have
\[
	\xi(s) = \xi(s+n -n) = \sigma(-n, \xi)(s+n) =\phi_n(s+n).
\]
Because $S$ is compact by Corollary \ref{cor_cpt}, we can construct a sequence $\{\sigma\qty(n_i,\phi_{n_i})\}_i$ that converges to $\psi \in S.$ Then, for all $t \in \mathbb{R},$ we have that
\[
\begin{aligned}
	\psi(t) &= \lim_{i\to \infty} \phi_{n_i}(t + n_i)\\
			&= \xi(t).
\end{aligned}
\]
Therefore, we have $\xi = \psi \in S.$
\end{proof}
\section{Invariant measures}\label{inv_meas}
In view of Theorem \ref{thm_cor}, every property regarding invariance induces a generalized counterpart through the evaluation-at-zero map. Furthermore, Theorem \ref{thm_g2y} implies that the set of all continuous complete trajectories is a solution set with good properties. In this section, we use these results as hints for defining invariant measures.

First, we recall the definition of invariant measures for a multivalued semigroup determined by a generalized semiflow, introduced in \cite{moussa2017invariant}. This concept can be formulated in several ways. Here, we introduce it in a form that is suitable for later discussion.
\begin{definition}[\cite{moussa2017invariant}]
Let $\{V(t)\}_{t \in \mathbb{R}_{\geq 0}}$ be a multivalued semigroup determined by a generalized semi flow $G$ that has continuous maps as trajectories. Let $\{V(t)\}_{t \in \mathbb{R}_{\geq 0}}$ have a continuous complete trajectory. Then, a Borel probability measure $\mu$ on $X$ is \textbf{invariant} if any one of the following conditions are satisfied:
\begin{enumerate}
\item For each Borel subset $A \subset X$ and $t \geq 0,$ $\mu(A) \leq \mu(V(t)^{-1} A).$
\item There exists a Borel probability measure $\nu$ on $C(\mathbb{R}, X)$, such that $\nu$ is invariant for the shift on $C(\mathbb{R},X)$, has support in the set of complete trajectories, and $\mu = (\pi_0)_*\nu,$ where $\pi_0$ is the evaluation at $0.$
\end{enumerate}
\end{definition}

Because the set of all continuous complete trajectories, which contains the support of invariant measures, is a solution set with good properties in terms of Yorke's axioms, we can generalize the definition above.
Considering that Yorke's formalism does not assume concatenation by default, we define the notion of invariant measures for a shift-invariant subset of $C(\mathbb{R}, X)$ as follows:
\begin{definition}\label{def_inv}
For a shift-invariant subset $S \subset C(\mathbb{R}, X),$ a Borel probability measure $\mu$ on $X$ is \textbf{invariant} if there exists a Borel probability measure $\nu$ on $S$ such that $\nu$ is invariant for the shift on $S$ and $\mu = (\pi)_*\nu$. 
\end{definition}

From Theorem \ref{thm_g2y}, we see that this definition generalizes the definition of Moussa. In other words, 
an invariant measure for the multivalued semigroup determined by a generalized semiflow is an invariant measure for $\mathcal{K}$ in the sense of Definition \ref{def_inv}.
\begin{example}
Let us consider the differential inclusion of $S^1$ defined by
\[
	\dot x \in [1,2].
\]
In this case, the Lebesgue measure is invariant because
\[
	A - t\subset V(t)^{-1} A 
\]
for all $A \subset S^1.$
\end{example}
The existence of an invariant measure can be established immediately using the classical Krylov-Bogoliubov theorem.
\begin{theorem}
Let $X$ be a compact metric space:
If $S \subset C(\mathbb{R}, X)$ satisfies the compactness axiom and is shift invariant, there exists an invariant measure for $S$ on $X.$
\end{theorem}
\begin{proof}
Because $(S, \sigma)$ is a continuous flow in compact space, there exists an invariant measure $\nu$. Then, measure $\mu := \pi_* \nu$ is invariant in the sense of Definition \ref{def_inv}.
\end{proof}
The next property enables us to generalize the results for invariant measures for a multivalued semigroup.
\begin{lemma}\label{lem_inv_ineq}
Let $\mu$ be an invariant measure of a shift-invariant subset $S \subset C(\mathbb{R}, X)$. For all $t \in \mathbb{R}$ and Borel set $A \subset X$, we have 
\[\mu(A) \leq \mu(V(t)^{-1} A).\]
\end{lemma}
\begin{proof}
 Let $\nu$ be invariant for the shift on $S$ and satisfy $\mu = (\pi)_*\nu$. For all $t \in \mathbb{R}$ and the Borel set $A \subset X$, we have
 \[
 	V(t)^{-1} A = V(-t) A = \pi\qty(\sigma(-t,S_A)).
 \]
 using Lemmas \ref{lem_smg} and \ref{lem_ss}. Then, we have
 \[
 	\begin{aligned}
		\mu(A) &= \nu(S_A)\\
				&= \nu\qty(\sigma(-t, S_A))\\
				&\leq \nu\qty(\pi^{-1}\qty(\pi\qty(\sigma(-t, S_A))))\\
				&=\mu\qty(V(t)^{-1} A).
	\end{aligned}
 \]
 This yields the required inequality.
\end{proof}
Now, we consider some generalizations of the classical results for the invariant measure. Lemma \ref{lem_inv_ineq} enables us to reuse existing proofs.
\begin{theorem}[Poincar\'e recurrence theorem, cf. Theorem 2.2 in \cite{aubin1991poincare}]
Let $\mu$ be an invariant measure of a shift-invariant subset $S \subset C(\mathbb{R}, X)$. Then, for all Borel subset $B \subset X,$ we have
\[
	\mu\qty(B \cap B_{\infty}) = \mu(B),
\]
where
\[
	B_\infty := \bigcap_{N=0}^\infty \bigcup_{n= N}^\infty V(n)^{-1}B.
\]
In particular, $\mu$-almost every point in $B$ returns to $B$ infinitely often. 
\end{theorem}
\begin{proof} Because we have $ V(n)^{-1} = V(-n) = V(1)^{-n}$ for all $n \in \mathbb{N}$, we can apply Theorem 2.2 in \cite{aubin1991poincare} to $F=V(1).$
\end{proof}
\begin{theorem}[Main Theorem D, cf. Theorem 2.10 in \cite{faure2013ergodic}]
Let $\mu$ be an invariant measure of a shift-invariant subset $S \subset C(\mathbb{R}, X)$. Then, the closure of the set of all recurrent points for $S$ is a full-measure set.
\end{theorem}
\begin{proof}
Let $A = \{x\in X \mid x \in \omega_S(x)\}$ and $R = \{\phi\in S \mid \phi \in \omega(\phi)\}$.
Once we have shown that $\pi\qty(\bar R) \subset \bar A$,
the remainder of the proof proceeds in a manner similar to that in Theorem 2.10 in \cite{faure2013ergodic}.

Let $x \in \pi\qty(\bar R)$. Then, we can find a convergent sequence $\phi_n \in R$ with $\phi_n(0) \to x$ as $n \to \infty$. Because $\phi_n \in \omega(S_{\phi_n(0)})$, $\phi_n(0) \in \omega_S\qty(\phi_n(0))$ for all $n$. Therefore, we can conclude that $x \in \bar A$.
\end{proof}
Finally, we comment on the difference in definition from the one proposed in \cite{novaes} for piecewise-smooth vector fields. The definition of the invariant measure in that study requires that, in our notation, $\mu(A) = \mu(V(t) A)$ for all $t \in \mathbb{R}$ and Borel set $A \subset X$, which is more restrictive compared to our definition. For example, the system considered in Example \ref{ex_fil} has an invariant measure with support on $y = 0$ in the sense of Definition \ref{def_inv}, whereas there is no invariant measure if we use the definition in \cite{novaes}.

\section{Concluding Remarks}\label{conc}
The results obtained thus far can be summarized as follows: as far as we consider Yorke's formalism on compact spaces assuming the global existence of solutions, we obtain generalized concepts regarding invariance directly through the evaluation-at-zero map, demonstrating properties similar to classical ones.

In particular, we may naturally define the ergodic invariant measure using that of the solution space; however, how it relates to the dynamics in the phase space is unclear. Given the possibility of applying ergodic theory to the study of differential inclusions or non-smooth dynamical systems, it is an interesting point to consider. 

Another interesting direction of research will be the application to the study of Filippov systems. While we have used Filippov systems as illustrating examples, this is, of course, a partial treatment. Recently, an orbit-space approach has been proposed to consider their dynamics \cite{gomide2023orbit}. Once we know how the framework presented here relates to it, a new perspective on the dynamics of Fillipov systems might be obtained.

\section*{Acknowledgements}
This study was supported by Grant-in-Aid for JSPS Fellows (20J01101). The author would like to thank the anonymous reviewers for valuable comments and suggestions.

\bibliographystyle{plain}
\bibliography{asym}

\end{document}